\documentclass[11pt]{amsart}
\title[ Lefschetz coincidence numbers of  solvmanifolds ]
{ Lefschetz coincidence numbers of solvmanifolds with Mostow conditions}

\author{Hisashi Kasuya}

\usepackage{amssymb}
\usepackage{amsmath}
\usepackage{amscd}
\usepackage{amstext}
\usepackage{amsfonts}
\usepackage[all]{xy}

\theoremstyle{plain}

\theoremstyle{plain}

\theoremstyle{plain}

\theoremstyle{plain}
\newtheorem{theorem}{Theorem}[section] 
\theoremstyle{remark}

\theoremstyle{Main result}
\newtheorem{main result}{Main result}
\theoremstyle{lemma}
\newtheorem{lemma}[theorem]{Lemma}
\theoremstyle{definition}
\newtheorem{definition}[theorem]{Definition}
\theoremstyle{proposition}
\newtheorem{proposition}[theorem]{Proposition}
\theoremstyle{corollary}

\theoremstyle{remark}

\address[Hisashi Kasuya]{Department of Mathematics, Tokyo Institute of Technology, 2-12-1 Ookayama, Meguro, Tokyo 152-8551, Japan}
\email{kasuya@math.titech.ac.jp}


\newcommand{\C}{\mathbb{C}}
\newcommand{\R}{\mathbb{R}}

\newcommand{\g}{\frak{g}}
\newcommand{\n}{\frak{n}}

\begin{document} 

\maketitle
\begin{abstract}
For any two continuous maps $f,g$ between two solvmanifolds of same dimension  satisfying the Mostow condition,
we give a technique of computation of the Lefschetz coincidence number of $f,g$.
This result is an extension of the result of Ha, Lee and Penninckx for completely solvable case.
\end{abstract}
\section{Introduction}

For two compact oriented manifolds $M_{1}$ and $M_{2}$ of the same dimension, for two continuous maps $f,g:M_{1}\to M_{2}$, as generalizations of the Lefschetz number and the Nielsen number for topological fixed point theory, the Lefschetz coincidence number $L(f,g)$ and the Nielsen coincidence number $N(f,g)$ are defined.
The Nielsen coincidence number $N(f,g)$ is a lower bound for the number of connected components of coincidences of $f$ and $g$.
But computing the Nielsen coincidence number is very difficult.
For some classes of manifolds, we have relationships between 
the Lefschetz coincidence number $L(f,g)$ and the Nielsen coincidence number $N(f,g)$.

Let $G$ be a simply connected solvable Lie group with a lattice (i.e. cocompact discrete subgroup of $G$) $\Gamma$.
We call $G/\Gamma$ a solvmanifold.
If $G$ is nilpotent, we call $G/\Gamma$ a nilmanifold.

For two solvmanifolds $G_{1}/\Gamma_{1}$ and $G_{2}/\Gamma_{2}$ with two continuous maps $f,g:G_{1}/\Gamma_{1}\to G_{2}/\Gamma_{2}$,
in \cite{Wa}, Wang showed the inequality 
\[\vert L(f,g)\vert \le N(f,g).
\]
Hence by Lefschetz coincidence number  $L(f,g)$ we can estimate the number of coincidences of $f,g$.
Suppose that $G_{1}$ and $G_{2}$ are completely solvable i. e.  for any element of $G$ the all eigenvalues of the adjoint operator of $g$ are real.
Then the de Rham cohomologies of solvmanifolds $G_{1}/\Gamma_{1}$ and $G_{2}/\Gamma_{2}$  are isomorphic to the cohomologies of the Lie algebras of  $G_{1}$ and $G_{2}$.
Moreover for the induced maps $f_{\ast},g_{\ast}:\pi_{1}(G_{1}/\Gamma_{1})\cong \Gamma_{1}\to \Gamma_{2}\cong \pi_{1}(G_{2}/\Gamma_{2})$, we can take homomorphisms $\Phi,\Psi:G_{1}\to G_{2}$ which are extensions of $f_{\ast},g_{\ast}$.
In \cite{HLP}, Ha, Lee and Penninckx computed the Lefschetz coincidence number  $L(f,g)$  by using "linearizations" $\Phi,\Psi$ of $f$ and $g$.

In this paper, for a solvmanifold $G/\Gamma$  we consider  the Mostow condition: "$ {\rm Ad}(G)$ and $ {\rm Ad}(\Gamma)$ have the same Zariski-closure in $ {\rm Aut}(\g_{\C})$" where ${\rm Ad}$ is the adjoint representation of a Lie group $G$.
The condition: "$G$ is completely solvable" is a special case of the Mostow condition (see \cite{Wi} and \cite{CF}). 
In \cite{Mosc}, Mostow showed that for a solvmanifold $G/\Gamma$ satisfying the Mostow condition,  the de Rham cohomology of  $G/\Gamma$ is  also isomorphic to the cohomology of the Lie algebra of  $G$.
However, for two solvmanifolds $G_{1}/\Gamma_{1}$ and $G_{2}/\Gamma_{2}$ satisfying the Mostow conditions,
extendability of homomorphisms between lattices $\Gamma_{1}$ and $\Gamma_{2}$ is not valid. 
(For isomorphisms, "virtually" extendability is known (\cite{Wi})).
Thus in order to compute  the Lefschetz coincidence number  $L(f,g)$ of   two continuous maps $f,g:G_{1}/\Gamma_{1}\to G_{2}/\Gamma_{2}$ between solvmanifolds satisfying  the Mostow condition, we should give new idea of "linearizations".

In this paper, we give a technique of  linearizations of all maps 
between solvmanifolds satisfying  the Mostow condition
and we give a formula for the Lefschetz coincidence number which is similar to the result by Ha, Lee and Penninckx (\cite{HLP}).

\section{Lefschetz numbers and spectral sequences}

Let $V^{\ast}$  be a finite dimensional graded vector space and $f^{\ast}:V^{\ast}\to V^{\ast}$ a graded linear map.
Then we denote \[L(f)=\sum_{i}(-1)^{i}{\rm tr}\, f^{i}.\]
\begin{lemma}\label{spelm}
Let $C^{\ast}$  be a bounded filtered cochain complex and $f^{\ast}:C^{\ast}\to C^{\ast}$ a morphism of filtered cochain complex with the induced map $H^{\ast}(f):H^{\ast}(C^{\ast})\to H^{\ast}(C^{\ast})$.
Consider the spectral sequences $E_{r}^{\ast,\ast}(C^{\ast})$  of $C^{\ast}$ 
and the map $E_{r}^{\ast,\ast}(f):E_{r}^{\ast,\ast}(C^{\ast})\to E_{r}^{\ast,\ast}(C^{\ast})$ induced by $f^{\ast}$.
Consider the graded linear map ${\rm Tot}^{\ast}E_{r}^{\ast,\ast}(f):{\rm Tot}^{\ast}E_{r}^{\ast,\ast}(C^{\ast})\to {\rm Tot}^{\ast}E_{r}^{\ast,\ast}(C^{\ast})$ for the total complex.
We suppose that for some integer $s$, for $r\ge s$, the $E_{r}$-term $E_{r}^{\ast,\ast}(C^{\ast})$ is finite dimensional.

Then for each $r\ge s$, we have
\[L(H^{\ast}(f))=L({\rm Tot}^{\ast}E_{r}^{\ast,\ast}(f)).
\]

\end{lemma}
\begin{proof}
By the assumption, sufficiently large $r$, we have
\[E_{r}^{p+q}(C)\cong F^{p}H^{p+q}(C)/F^{p+1}H^{p+q}(C).
\]
Hence by using the property of trace (see \cite[Proposition 2.3.11]{JM}) we have
\[\sum_{p+q=k}{\rm tr}\, E^{p,q}_{r}(f)={\rm tr} H^{k}(f).
\]
By the Hopf lemma for  trace (see \cite[Lemma 2.3.23]{JM}),
we have 
\[\sum_{p,q}(-1)^{p+q}{\rm tr}\, E^{p,q}_{r}(f)=\sum_{p,q}(-1)^{p+q}{\rm tr}\, E^{p,q}_{r-1}(f).
\]
and inductively for $s\le r$, we have
\[\sum_{p,q}(-1)^{p+q}{\rm tr}\, E^{p,q}_{r}(f)=\sum_{p,q}(-1)^{p+q}{\rm tr}\, E^{p,q}_{s}(f).
\]
Hence the lemma follows.
\end{proof}

Let $A^{\ast}$ be a finite-dimensional graded commutative $\C$-algebra.
\begin{definition}
$A^{\ast}$ is of degree $n$ Poincar\'e duality type (n-PD-type) if the following conditions hold:

$\bullet$ $A^{\ast<0}=0$ and $A^{0}=\R 1$ where $1$ is the identity element of $A^{\ast}$.

$\bullet$ For some positive integer $n$, $A^{\ast>n}=0$ and $A^{n}=\R v$ for $v\not=0$.

$\bullet$ For any $0<i<n$ the bi-linear map $A^{i}\times A^{n-i}\ni (\alpha,\beta)\mapsto \alpha\cdot \beta\in A^{n}$ is non-degenerate.
Hence we have an isomorphism $D_{i}:A^{n-i}\cong (A^{i})^{\ast}$ where $(A^{i})^{\ast}$ is the dual space of $A^{i}$.

\end{definition}

Let $A^{\ast}_{1}$ and $A_{2}^{\ast}$ be  finite-dimensional graded commutative $\R$-algebras of n-PD-type and $f^{\ast}:A_{2}^{\ast}\to A_{1}^{\ast}$ and $g^{\ast}:A_{2}^{\ast}\to A_{1}^{\ast}$  graded linear maps.
By isomorphisms
$:A_{1}^{i}\cong (A_{1}^{n-i})^{\ast}$ and $:A_{2}^{i}\cong (A_{2}^{n-i})^{\ast}$, we have the map $D^{i}(g^{\ast}):A_{1}^{i}\to A_{2}^{i}$ which corresponds to the dual map $(A_{1}^{n-i})^{\ast}\to (A_{2}^{n-i})^{\ast}$ of $g^{n-i}$.
Define the map $\theta^{i}(f,g)=D^{i}(g^{\ast})\circ f^{i}$.
We denote \[L(f,g)=L\left(\theta^{i}(f,g)\right).\]

For two compact oriented manifolds $M_{1}$ and $M_{2}$ of the same dimension, for two continuous maps $f,g:M_{1}\to M_{2}$, we consider the induced maps
$H^{\ast}(f), H^{\ast}(g): H^{\ast}(M_{2})\to H^{\ast}(M_{1})$.
Then the Lefschetz coincidence number $L(f,g)$ is defined as
$L(f,g)=L\left(H^{\ast}(f), H^{\ast}(g)\right)$.

\begin{definition}
 A differential graded algebra (DGA) is a graded  commutative $\R$-algebra $A^{\ast}$  with a differential $d$ of degree +1 so that $d\circ d=0$ and $d(\alpha\cdot \beta)=d\alpha\cdot \beta+(-1)^{p}\alpha\cdot d\beta$ for $\alpha\in A^{p}$.
\end{definition}

\begin{definition}
A finite-dimensional DGA $(A^{\ast},d)$ is of $n$-PD-type if  the following conditions hold:

$\bullet$  $A^{\ast}$ is  a finite-dimensional graded $\R$-algebra of $n$-PD-type.

$\bullet$ $dA^{n-1}=0$ and $dA^{0}=0$.
\end{definition}

As similar to the  Poincar\'e duality of the cohomology of compact Riemannian manifold, we can prove the following lemma.

\begin{lemma}[\cite{KSP}]\label{rm}
Let $(A^{\ast},d)$ be a finite dimensional DGA of $n$-PD-type.
Then the cohomology algebra $H^{\ast}(A)$ is   a finite dimensional graded commutative $\R$-algebra of $n$-PD-type.
\end{lemma}

Then the following lemma follows from Lemma \ref{rm} inductively.

\begin{lemma}\label{fffi}
Let $A^{\ast}$ be a bounded filtered differential graded algebra.
Suppose  that:
\begin{itemize}
\item The cohomology $H^{\ast}(A^{\ast})$ is a finite dimensional graded algebra of $n$-PD-type.
\item For some integer $s$,   the total complex  $({\rm Tot}^{\ast}\,E_{s}^{\ast,\ast}(A^{\ast}), d_{s})$ of the  $E_{s}$-term of the spectral sequence is a finite dimensional graded algebra of $n$-PD-type.
\end{itemize}
Then for each $r\ge s$, the total complex  $({\rm Tot}^{\ast}\,E_{r}^{\ast,\ast}(\g), d_{r})$ of the  $E_{r}$-term of the spectral sequence is also a graded algebra of $n$-PD-type.
\end{lemma}
\begin{proof}
Since we have $H^{0}(A^{\ast})\cong \R$, $H^{n}(A^{\ast})\cong \R$, ${\rm Tot}^{0}\,E_{s}^{\ast,\ast}(A^{\ast})\cong \R$ and ${\rm Tot}^{n}\,E_{s}^{\ast,\ast}(A^{\ast})\cong \R$, 
we have $d_{s}( {\rm Tot}^{0}\,E_{s}^{\ast,\ast}(A^{\ast}))=0$ and $d_{s}( {\rm Tot}^{n-1}\,E_{s}^{\ast,\ast}(A^{\ast}))=0$.
Hence  the total complex  $({\rm Tot}^{\ast}\,E_{s}^{\ast,\ast}(A^{\ast}), d_{s})$ of the  $E_{s}$-term  is  a DGA of $n$-PD-type and  by Lemma \ref{rm}, the total complex  ${\rm Tot}^{\ast}\,E_{s+1}^{\ast,\ast}(A^{\ast})$ is  a graded algebra of $n$-PD-type.
\end{proof}

By Lemma \ref{spelm}, we have:
\begin{lemma}\label{spel}
Let $A^{\ast}_{1}$ and $A_{2}^{\ast}$ be bounded filtered DGAs and $f^{\ast}, g^{\ast}:A_{2}^{\ast}\to A_{1}^{\ast}$  morphisms of filtered DGA with the induced maps $H^{\ast}(f), H^{\ast}(g):H^{\ast}(A_{2}^{\ast})\to H^{\ast}(A_{1}^{\ast})$.
Consider the spectral sequences $E_{r}^{\ast,\ast}(A_{1})$ and $E_{r}^{\ast,\ast}(A_{2})$ of $A^{\ast}_{1}$ and $A_{2}^{\ast}$ 
and the maps $E_{r}^{\ast,\ast}(f), E_{r}^{\ast,\ast}(g):E_{r}^{\ast,\ast}(A_{2})\to E_{r}^{\ast,\ast}(A_{1})$ induced by $f,g$.

We suppose that:
\begin{itemize}
\item The cohomologies $H^{\ast}(A^{\ast}_{1})$ and $H^{\ast}(A^{\ast}_{2})$ are finite dimensional graded algebra of $n$-PD-type.
\item For some integer $s$,  the total complexes ${\rm Tot}^{\ast}\,E_{r}^{\ast,\ast}(A_{1})$ and ${\rm Tot}^{\ast}\,E_{r}^{\ast,\ast}(A_{2})$ of $E_{r}$-terms are finite dimensional graded algebras of $n$-PD-type.
Hence inductively the lemma follows.
\end{itemize}
Then for each $r\ge s$, we have
\[L(H^{\ast}(f), H^{\ast}(g))=L({\rm Tot}^{\ast}E_{r}^{\ast,\ast}(f), {\rm Tot}^{\ast}E_{r}^{\ast,\ast}(g)).
\]

\end{lemma}

\section{The Ha-Lee-Penninckx formula}
Let $V$ be a $n$-dimensional vector space.
Consider the exterior algebra $\bigwedge V$.
Then $\bigwedge V$ is a  finite-dimensional graded commutative $\C$-algebras of n-PD-type.
In \cite{HLP}, Ha-Lee-Penninckx showed:
\begin{theorem}[\cite{HLP}]\label{HL}
Let $V_{1}$, $V_{2}$ be  $n$-dimensional vector spaces and $\Phi, \Psi:V_{2}\to V_{1}$ linear maps.
Consider the exterior algebras $\bigwedge V_{1}$ and $\bigwedge V_{2}$ and the extended map
$\wedge \Phi, \wedge \Psi:\bigwedge V_{2}\to \bigwedge V_{1}$.
Take representation matrices $A$, $B$ of  $\Phi$ and $\Psi $ associated with basis of $V_{1}$ and $V_{2}$.
Then we have
\[L(\wedge \Phi, \wedge \Psi)={\rm det}(A-B).
\]
\end{theorem}

\section{Lie algebra cohomology}
Let $\g$ be a $n$-dimensional solvable  Lie algebra.
We consider the DGA $\bigwedge \g^{\ast}$  with the differential $d$ which is the dual to the Lie bracket of $\g$.
We suppose that $\g$ is unimodular.
Then $\bigwedge \g^{\ast}$ is a DGA of $n$-PD-type.
Take a basis $X_{1},\dots,\ X_{n}$ of $\g$ and its dual basis $x_{1},\dots , x_{n}$ of $\g^{\ast}$.

Let $\frak n$ be a ideal of $\g$.
We consider the spectral sequence $(E_{r}^{p,q}(\g), d_{r})$ given by the extension $0\to \frak n\to \g\to \g/\frak n\to 0$.
This  spectral sequence is given by the filtration
\[F^{p}\bigwedge ^{p+q}\g^{\ast}=\{\omega\in \bigwedge^{p+q} \g^{\ast}\vert \omega(Y_{1},\dots, Y_{p+1})=0 \,\, {\rm for}\,\, Y_{1},\dots, Y_{p+1}\in\n\}.
\]
We have 
\[E_{0}^{\ast,\ast}(\g)=\bigwedge (\g/\frak n)^{\ast}\otimes \bigwedge\frak n^{\ast}\]
with the differential $d_{0}= 1\otimes d_{\bigwedge\frak n^{\ast}}$,
\[E_{1}^{\ast,\ast}(\g)=\bigwedge (\g/\frak n)^{\ast}\otimes H^{\ast}(\frak n)\]
whose differential $d_{1}$ is the differntial on $\bigwedge (\g/\frak n)^{\ast}\otimes H^{\ast}(\frak n)$ twisted by the  action of $\g/\n$ on $ H^{\ast}(\frak n)$
and 
\[E_{2}^{\ast,\ast}(\g)=H^{\ast}\left(\g, H^{\ast}(\frak n)\right).\]


Since we suppose that $\g$ is unimodular, we have $d\left(\bigwedge^{n-1} \g^{\ast}\right)=0 $ and so  $\bigwedge \g^{\ast}$ is a finite dimensional DGA of $n$-PD-type.
By Lemma \ref{fffi},  the total complex  $({\rm Tot}^{\ast}\,E_{r}^{\ast,\ast}(\g), d_{r})$ of each  $E_{r}$-term of the spectral sequence is also a graded algebra of $n$-PD-type.

\section{de Rham Cohomology  solvamanifolds with Mostow conditions}\label{cos}

Let $G$ be a simply connected solvable Lie group with a lattice $\Gamma$.
We suppose the Mostow condition: $ {\rm Ad}(G)$ and $ {\rm Ad}(\Gamma)$ have the same Zariski-closure in $ {\rm Aut}(\g_{\C})$.
Then we have:
\begin{proposition}[\cite{OB}]\label{2lat}
Discrete subgroups $[\Gamma, \Gamma]$ and $\Gamma\cap [G,G]$ are lattices in the Lie group $[G,G]$ and the subgroup $\Gamma[G,G]$ is closed in $G$.
\end{proposition}

Set $[G,G]=N$, $G/N=A$ and $\frak n$ the Lie algebra of $N$ and $\frak a$ the Lie algebra of $A$.
By Proposition \ref{2lat}, we have the fiber bundle structure
\[N/\Gamma\cap N\to G/\Gamma\to G/\Gamma N
\]
of the solvmanifold $ G/\Gamma$  with base space torus $G/\Gamma N=A/p(\Gamma)$  and fiber nilmanifold  $N/\Gamma\cap N$ where $p:G\to G/N$ is the quotient map.

We consider the filtration 
\[F^{p}\bigwedge ^{p+q}\g^{\ast}=\{\omega\in \bigwedge^{p+q} \g^{\ast}\vert \omega(X_{1},\dots, X_{p+1})=0 \,\, {\rm for}\,\, X_{1},\dots, X_{p+1}\in \n\}.
\]
This filtration gives the filtration of the cochain complex $\bigwedge \g^{\ast}$ and the filtration of the de Rham complex $A^{\ast}(G/\Gamma)$.
We consider the spectral sequence $E^{\ast,\ast}_{\ast}(\g)$ of $\bigwedge \g^{\ast}$ and the spectral sequence $E^{\ast,\ast}_{\ast}(G/\Gamma)$ of $A^{\ast}(G/\Gamma)$.
Then we have the commutative diagram
\[\xymatrix{
	E_{2}^{\ast,\ast}(\g)\ar[r]\ar[d]^{\cong}& E_{2}^{\ast,\ast}(G/\Gamma)\ar[d]^{\cong} \\
	 H^{\ast}\left(\frak a, H^{\ast}(\frak n)\right) \ar[r]&H^{\ast}
\left(A/p(\Gamma),{\bf H}^{\ast}(N/\Gamma\cap N )\right)
 }
\]
where ${\bf H}^{\ast}(N/\Gamma\cap N )$ is the local system on the cohomology of fiber induced by the fiber bundle
(see \cite{Hatt}, \cite[Section 7]{R}).

\begin{theorem}\label{e2}
The induced map $E_{2}^{\ast,\ast}(\g)\to E_{2}^{\ast,\ast}(G/\Gamma)$ is an isomorphism.

\end{theorem}
\begin{proof}
We first show that for each $r$, the induced map $E_{r}^{\ast,\ast}(\g)\to E_{r}^{\ast,\ast}(G/\Gamma)$ is injective.
A simply connected solvable Lie group with a lattice is unimodular (see \cite[Remark 1.9]{R}).
Let $d\mu$ be a bi-invariant volume form such that $\int_{G/\Gamma}d\mu =1$.
For $\alpha\in A^{p}(G/\Gamma)$, we have a left-invariant form $\alpha_{inv}\in \bigwedge^{p} \g^{\ast}$ defined  by
\[\alpha_{inv}(X_{1},\dots ,X_{p})=\int_{G/\Gamma}\alpha(\tilde X_{1},\dots ,\tilde X_{p})d\mu
\]
for $X_{1},\dots ,X_{p}\in \g$ where $\tilde X_{1},\dots ,\tilde X_{p}$ are vector fields on $G/\Gamma$ induced by $X_{1},\dots X_{p}$.
We define the map $I:A^{\ast}(M)\to  \bigwedge \g^{\ast}$ by $\alpha \mapsto \alpha_{inv}$.
Then this map is a cochain complex map (see \cite{KVai}) such
 that $I\circ i={\rm id}_{ \vert_{
\bigwedge \g^{\ast}}}$.
The map $I$ is compatible with the filtration as above.
Hence $I$ induces a homomorphism $E_{r}^{\ast,\ast}(G/\Gamma)\to E_{r}^{\ast,\ast}(\g)$.
This implies that the induced map  $E_{r}^{\ast,\ast}(\g)\to E_{r}^{\ast,\ast}(G/\Gamma)$ is injective.

Consider the $A$-action on $H^{\ast}(\n)$ which is the extension of the $\frak a$-action on $H^{\ast}(\n)$ given by  $0\to \frak n\to \g\to \frak a\to 0$.
Since we have $H^{\ast}(\n)\cong H^{\ast}(N/\Gamma\cap N)$.
The local system ${\bf H}^{\ast}(N/\Gamma\cap N )$ is given by the $\Gamma$-action on $H^{\ast}(\n)$ which is the restriction of  the $A$-action on $H^{\ast}(\n)$.
Since $ {\rm Ad}(G)$ and $ {\rm Ad}(\Gamma)$ have the same Zariski-closure in $ {\rm Aut}(\g_{\C})$, the images of actions $A\to {\rm Aut}(H^{\ast}(\n))$ and $p(\Gamma)\to {\rm Aut}(H^{\ast}(\n))$  have  also the same Zariski-closure in $ {\rm Aut}(H^{\ast}(\n))$.
Then by \cite[Theorem 7.26]{R}
we have
\[H^{\ast}\left(\frak a, H^{\ast}(\frak n)\right)\cong
H^{\ast}
\left(A/p(\Gamma),{\bf H}^{\ast}(N/\Gamma\cap N )\right)
\]
Hence the theorem follows.
\end{proof}

\section{Linearizations of  solvamanifolds with Mostow conditions}\label{linnn}

Consider two simply connected solvable Lie groups $G_{1}$ and $G_{2}$ with lattices $\Gamma_{1}$ and $\Gamma_{2}$.
We assume that they satisfy the Mostow condition.
Let $\phi : \Gamma_{1}\to \Gamma_{2}$ be a homomorphism.
Then we have 
\[\phi([\Gamma_{1},\Gamma_{1}])\subset [\Gamma_{2},\Gamma_{2}].\]
Hence  $\phi$ induces the homomorphism $\bar\phi:\Gamma_{1}/[\Gamma_{1},\Gamma_{1}]\to \Gamma_{2}/[\Gamma_{2},\Gamma_{2}]$.
We show
\begin{lemma}\label{com}
$\phi(\Gamma_{1}\cap [G_{1},G_{1}])\subset \Gamma_{2}\cap [G_{2},G_{2}]$. 
\end{lemma}

\begin{proof}
Consider the surjection 
\[\Gamma_{1}/[\Gamma_{1},\Gamma_{1}]\ni (g \,\,\, {\rm  mod}\,\,[\Gamma_{1},\Gamma_{1}]) \,  \mapsto (g \,\,\, {\rm  mod} \,\, \Gamma_{1}\cap [G_{1},G_{1}])\,\,\in \Gamma/\Gamma_{1}\cap [G_{1},G_{1}].\]
By Proposition \ref{2lat}, two nilpotent groups $[\Gamma_{1},\Gamma_{1}]$ and $\Gamma_{1}\cap [G_{1},G_{1}]$  have same rank and hence the kernel of this surjection consists of torsions.
This implies that for $g\in \Gamma_{1}\cap [G_{1},G_{1}]$, the element 
\[\bar\phi(g \,\,\, {\rm  mod}\,\,[\Gamma_{1},\Gamma_{1}])=\phi(g) \,\,\, {\rm  mod}\,\,[\Gamma_{2},\Gamma_{2}]\]
 is a torsion.
Since the group $\Gamma_{2}/\Gamma_{2}\cap [G_{2},G_{2}]$ is a lattice in $G_{2}/[G_{2},G_{2}]$, $\Gamma_{2}/\Gamma_{2}\cap [G_{2},G_{2}]$ is torsion-free.
Hence we have 
\[(\phi(g)\,\,\, {\rm  mod} \,\, \Gamma_{2}\cap [G_{2},G_{2}])=(0\,\,\, {\rm  mod} \,\, \Gamma_{2}\cap [G_{2},G_{2}])\]
 for $g\in \Gamma_{1}\cap [G_{1},G_{1}]$.
Thus the lemma follows.
\end{proof}

Set $N_{1}=[G_{1},G_{1}]$, $N_{2}=[G_{2},G_{2}]$, $A_{1}=G_{1}/N_{1}$ and $A_{2}=G_{2}/N_{2}$.
Let $\n_{1}$, $\n_{2}$, $\frak a_{1}$ and $\frak a_{2}$ be the Lie algebras of $N_{1}$, $N_{2}$, $A_{1}$ and $A_{2}$ respectively.
Consider the quotient maps $p_{1}:G_{1}\to A_{1}$ and   $p_{2}:G_{2}\to A_{2}$.
By Lemma \ref{com}, we have the commutative diagram
\[\xymatrix{
1\ar[r]&\Gamma_{1}\cap N_{1}\ar[r]\ar[d]^{\phi}&\Gamma_{1}\ar[r]\ar[d]^{\phi}&p_{1}(\Gamma_{1})\ar[r]\ar[d]^{\bar\phi}&1\\
1\ar[r]&\Gamma_{2}\cap N_{2}\ar[r]&\Gamma_{2}\ar[r]&p_{2}(\Gamma_{2})\ar[r]&1
}
\]

Since $\Gamma_{1}\cap N_{1}$, $\Gamma_{2}\cap N_{2}$,  $p_{1}(\Gamma_{1})$ and $p_{2}(\Gamma_{2})$ are lattices in $N_{1}$, $N_{2}$, $A_{1}$ and $A_{2}$ respectively,
we can take unique Lie group homomorphisms $\Phi_{1}:N_{1}\to N_{2}$ and $\Phi_{2}:A_{1}\to A_{2}$ which are extensions of $\phi:\Gamma_{1}\cap N_{1}\to \Gamma_{2}\cap N_{2}$ and $\bar\phi: p_{1}(\Gamma_{1})\to p_{2}(\Gamma_{2})$.

\begin{lemma}\label{lni}

We consider the spectral sequences
\[E_{0}^{\ast,\ast}(\g_{1})=\bigwedge\frak a_{1}^{\ast}\otimes \bigwedge \n_{1}^{\ast},\]
\[E_{0}^{\ast,\ast}(\g_{2})=\bigwedge\frak a_{2}^{\ast}\otimes \bigwedge \n_{2}^{\ast}\]
and
\[E_{1}^{\ast,\ast}(\g_{1})=\bigwedge\frak a_{1}^{\ast}\otimes  H^{\ast}( \n_{1}),\]
\[E_{1}^{\ast,\ast}(\g_{2})=\bigwedge\frak a_{2}^{\ast}\otimes H^{\ast} (\n_{2})\]

Then the linear map
\[\wedge\Phi_{2}^{\ast}\otimes\wedge\Phi_{1}^{\ast}:E_{0}^{\ast,\ast}(\g_{2})= \bigwedge\frak a_{2}^{\ast}\otimes \bigwedge \n^{\ast}_{2}\to \bigwedge\frak a_{1}^{\ast}\otimes \bigwedge \n^{\ast}_{1}=E_{0}^{\ast,\ast}(\g_{1})\]
is a cochain complex map and induced map
\[\wedge\Phi_{2}^{\ast}\otimes H^{\ast}(\wedge\Phi_{1}^{\ast}):E_{1}^{\ast,\ast}(\g_{2})= \bigwedge\frak a_{2}^{\ast}\otimes H^{\ast} (\n_{2})\to\bigwedge\frak a_{1}^{\ast}\otimes  H^{\ast}( \n_{1})=E_{1}^{\ast,\ast}(\g_{1})\]
is a cochain complex map.

\end{lemma}
\begin{proof}
Since $\Phi_{1}$ is a homomorphism of Lie group,
 the linear map
\[\wedge\Phi_{2}^{\ast}\otimes\wedge\Phi_{1}^{\ast}:E_{0}^{\ast,\ast}(\g_{2})=\bigwedge\frak a_{2}^{\ast}\otimes \bigwedge \n^{\ast}_{2}\to \bigwedge\frak a_{1}^{\ast}\otimes \bigwedge \n^{\ast}_{1}=E_{0}^{\ast,\ast}(\g_{1})\]
is cochain complex map.
We consider the induced map
\[\wedge\Phi_{2}^{\ast}\otimes H^{\ast}(\wedge\Phi_{1}^{\ast}):E_{1}^{\ast,\ast}(\g_{2})= \bigwedge\frak a_{2}^{\ast}\otimes H^{\ast} (\n_{2})\to\bigwedge\frak a_{1}^{\ast}\otimes  H^{\ast}( \n_{1})=E_{1}^{\ast,\ast}(\g_{1}).\]
We show that this map is a cochain complex homomophism.

We consider the group cohomologies $H^{\ast}(\Gamma_{1}\cap N_{1}, \R)$ and $H^{\ast}(\Gamma_{2}\cap N_{2},\R)$ and the induced map $H^{\ast}(\phi):H^{\ast}(\Gamma_{2}\cap N_{2},\R)\to H^{\ast}(\Gamma_{1}\cap N_{1},\R)$ of $\phi: \Gamma_{1}\cap N_{1}\to \Gamma_{2}\cap N_{2}$.
By the commutative diagram
\[\xymatrix{
1\ar[r]&\Gamma_{1}\cap N_{1}\ar[r]\ar[d]^{\phi}&\Gamma_{1}\ar[r]\ar[d]^{\phi}&p_{1}(\Gamma_{1})\ar[r]\ar[d]^{\bar\phi}&1\\
1\ar[r]&\Gamma_{2}\cap N_{2}\ar[r]&\Gamma_{2}\ar[r]&p_{2}(\Gamma_{2})\ar[r]&1,
}
\]
for the $p_{1}(\Gamma_{1})$-action $\delta_{1}:p_{1}(\Gamma_{1})\to {\rm Aut}(H^{\ast}(\Gamma_{1}\cap N_{1}, \R))$ and the $p_{2}(\Gamma_{2})$-action $\delta_{2}:p_{2}(\Gamma_{2})\to {\rm Aut}(H^{\ast}(\Gamma_{2}\cap N_{2}, \R))$,
we have
\[H^{\ast}(\phi)\circ \delta_{2}( \bar\phi(g))= \delta_{1}(g)\circ H^{\ast}(\phi).\]

By the isomorphisms,
\[H^{\ast}(\Gamma_{1}\cap N_{1}, \R)\cong H^{\ast}(N_{1}/\Gamma_{1}\cap N_{1},\R)\cong H^{\ast}(\n_{1})
\]
and 
\[H^{\ast}(\Gamma_{2}\cap N_{2}, \R)\cong H^{\ast}(N_{2}/\Gamma_{2}\cap N_{2},\R)\cong H^{\ast}(\n_{2}),
\]
we have $H^{\ast}(\phi)=H^{\ast}(\Phi_{1})$.
Consider the $A_{1}$-action $\Delta_{1}:A\to {\rm Aut}(H^{\ast}(\n_{1}))$ induced by the extension
$1\to N_{1}\to G_{1}\to A_{1}\to 1$
 and  $A_{2}$-action $\Delta_{2}:A\to {\rm Aut}(H^{\ast}(\n_{2}))$ induced by the extension
$1\to N_{2}\to G_{2}\to A_{2}\to 1$.
By $H^{\ast}(\phi)=H^{\ast}(\Phi_{1})$ and $H^{\ast}(\phi)\circ \delta_{2}( \bar\phi(g))= \delta_{1}(g)\circ H^{\ast}(\phi)$, we have
\[H^{\ast}(\Phi_{1})\circ \Delta_{2}( \Phi_{2}(v))=\Delta_{1}(v)\circ H^{\ast}(\Phi_{1})
\]
for all $v\in p(\Gamma_{1})\subset A_{1}$.
By the Mostow condition,  $\Delta_{1}(A_{1}) \times\Delta_{2}(\Phi_{2}(A_{2}))$ and $\Delta_{1}(p_{1}(\Gamma_{1}))\times \Delta_{2}(\Phi_{2}(p_{2}(\Gamma_{2})))$ have the same Zariski-closure in ${\rm Aut}(H^{\ast}(\n_{1}))\times {\rm Aut}(H^{\ast}(\n_{2}))$.
By this we have 
\[H^{\ast}(\Phi_{1})\circ \Delta_{2}( \Phi_{2}(v))=\Delta_{1}(v)\circ H^{\ast}(\Phi_{1})
\]
for all $v\in A_{1}$.

Consider the Lie algebra homomorphism $ \Phi_{2\ast}:\frak a_{1}\to \frak a_{2}$ and the $\frak a_{1}$-action $\Delta_{1\ast}:\frak a_{1}\to {\rm End}(H^{\ast}(\n_{1}))$ 
 and  $\frak a_{2}$-action $\Delta_{2\ast}:\frak a_{2\ast}\to {\rm End}(H^{\ast}(\n_{2}))$.
Then we have 
\[H^{\ast}(\Phi_{1})\circ \Delta_{2\ast}( \Phi_{2\ast}(V))=\Delta_{1\ast}(V)\circ H^{\ast}(\Phi_{1})
\]
for all $V\in \frak a_{1}$.
This implies that the map
\[\wedge\Phi_{2}^{\ast}\otimes H^{\ast}(\wedge\Phi_{1}^{\ast}):E_{1}^{\ast,\ast}(\g_{2})= \bigwedge\frak a_{2}^{\ast}\otimes H^{\ast} (\n_{2})\to\bigwedge\frak a_{1}^{\ast}\otimes  H^{\ast}( \n_{1})=E_{1}^{\ast,\ast}(\g_{1}).\]
is a cochain complex homomophism,
since the differentials of the  cochain complexes $E_{1}^{\ast,\ast}(\g_{1})= \bigwedge\frak a_{1}^{\ast}\otimes H^{\ast}( \n_{1})$ and $E_{1}^{\ast,\ast}(\g_{2})=\bigwedge\frak a_{2}^{\ast}\otimes  H^{\ast}( \n_{2})$ are twisted by the $\frak a_{1}$-action $\Delta_{1\ast}:\frak a_{1}\to {\rm End}(H^{\ast}(\n_{1}))$ 
 and  the $\frak a_{2}$-action $\Delta_{2\ast}:\frak a_{2\ast}\to {\rm End}(H^{\ast}(\n_{2}))$ respectively.
\end{proof}

Let $f:G_{1}/\Gamma_{1}\to G_{2}/\Gamma_{2}$ be a continuous map.
We consider the induced map $f_{\ast}:\pi_{1}(G_{1}/\Gamma_{1})\cong \Gamma_{1}\to \Gamma_{2}\cong G_{2}/\Gamma_{2}$.
We write $\phi=f_{\ast}$.
In this case,  the pair $\Phi_{1}$,  $\Phi_{2} $ constructed as above is called the linearlization of $f$.
Consider the spectral sequences
$E_{r}^{\ast,\ast}(G_{1}/\Gamma_{1})$ and $E_{r}^{\ast,\ast}(G_{2}/\Gamma_{2})$
as Section \ref{cos}.
Then for $r\ge 2$, $E_{r}^{\ast,\ast}(G_{1}/\Gamma_{1})$ and $E_{r}^{\ast,\ast}(G_{2}/\Gamma_{2})$ are identified with the Leray-Serre spectral sequences.
By commutative diagram
\[\xymatrix{
1\ar[r]&\Gamma_{1}\cap N_{1}\ar[r]\ar[d]^{\phi}&\Gamma_{1}\ar[r]\ar[d]^{\phi}&p_{1}(\Gamma_{1})\ar[r]\ar[d]^{\bar\phi}&1\\
1\ar[r]&\Gamma_{2}\cap N_{2}\ar[r]&\Gamma_{2}\ar[r]&p_{2}(\Gamma_{2})\ar[r]&1,
}
\]
Any continous map from $G_{1}/\Gamma_{1}$ to $G_{2}/\Gamma_{2}$ is homotopic to a continous map $f:G_{1}/\Gamma_{1}\to G_{2}/\Gamma_{2}$ which is a fiber-preserving map as
\[\xymatrix{
1\ar[r]&N_{1}/\Gamma_{1}\cap N_{1}\ar[r]\ar[d]^{f}&G_{1}/\Gamma_{1}\ar[r]\ar[d]^{f}&A_{1}/p_{1}(\Gamma_{1})\ar[r]\ar[d]^{\bar f}&1\\
1\ar[r]&N_{2}/\Gamma_{2}\cap N_{2}\ar[r]&G_{2}/\Gamma_{2}\ar[r]&A_{2}/p_{2}(\Gamma_{2})\ar[r]&1.
}
\]
Consider the induced map $E^{\ast,\ast}_{r}(f): E_{r}^{\ast,\ast}(G_{1}/\Gamma_{1})\to E_{r}^{\ast,\ast}(G_{2}/\Gamma_{2})$.
Then 
\[E^{\ast,\ast}_{2}(f):H^{\ast}
\left(A_{2}/p(\Gamma_{2}),{\bf H}^{\ast}(N_{2}/\Gamma_{2}\cap N_{2} )\right)\to H^{\ast}
\left(A_{1}/p(\Gamma_{1}),{\bf H}^{\ast}(N_{1}/\Gamma_{1}\cap N_{1} )\right)
\]
is induced by the fiber map $f : N_{1}/\Gamma_{1}\cap N_{1}\to N_{2}/\Gamma_{2}\cap N_{2}$ and the  base space map  $\bar f :A_{1}/p(\Gamma_{1})\to A_{2}/p(\Gamma_{2})$ (see \cite{mc}).
Consider the linearlization $\Phi_{1}$,  $\Phi_{2} $  of $f$ and induced maps $\underline{ \Phi_{1}}:N_{1}/\Gamma_{1}\cap N_{1}\to N_{2}/\Gamma_{2}\cap N_{2}$ and $\underline{ \Phi_{2}} :A_{1}/p(\Gamma_{1})\to A_{2}/p(\Gamma_{2})$.
Then the fiber map $f : N_{1}/\Gamma_{1}\cap N_{1}\to N_{2}/\Gamma_{2}\cap N_{2}$ and the  base space map  $\bar f :A_{1}/p(\Gamma_{1})\to A_{2}/p(\Gamma_{2})$ are homotopic to $\underline{ \Phi_{1}}:N_{1}/\Gamma_{1}\cap N_{1}\to N_{2}/\Gamma_{2}\cap N_{2}$ and $\underline{ \Phi_{2}} :A_{1}/p(\Gamma_{1})\to A_{2}/p(\Gamma_{2})$ respectively.
By Theorem \ref{e2}, we have 
\[H^{\ast}(\frak a_{1}, H^{\ast}(\n_{1}))\cong H^{\ast}
\left(A_{1}/p(\Gamma_{1}),{\bf H}^{\ast}(N_{1}/\Gamma_{1}\cap N_{1} )\right)\]
and
\[H^{\ast}(\frak a_{2}, H^{\ast}(\n_{2}))\cong H^{\ast}
\left(A_{2}/p(\Gamma_{2}),{\bf H}^{\ast}(N_{2}/\Gamma_{2}\cap N_{2} )\right).\]
By these isomorphisms, $E^{\ast,\ast}_{2}(f)$ is induced by $\wedge\Phi^{\ast}_{1}:\bigwedge \n^{\ast}_{2}\to \bigwedge \n^{\ast}_{1}$ and $\wedge\Phi^{\ast}_{2}:\bigwedge \frak a^{\ast}_{2}\to \bigwedge \frak a^{\ast}_{1}$.
Hence by Lemma \ref{lni} we have:
\begin{lemma}\label{sem}
The map
\[E_{2}(f): E_{2}^{\ast,\ast}(G_{2}/\Gamma_{2})\to E_{2}^{\ast,\ast}(G_{1}/\Gamma_{1})\]
is identified with the map
\[H^{\ast}(\wedge\Phi_{2}^{\ast})\otimes H^{\ast}(\wedge\Phi_{1}^{\ast}):E_{2}^{\ast,\ast}(\g_{2})= H^{\ast}\left(\frak a_{1}, H^{\ast}(\frak n_{2})\right)\to H^{\ast}\left(\frak a_{1}, H^{\ast}(\frak n_{1})\right)=E_{2}^{\ast,\ast}(\g_{1})\]
induced by
the  cochain complex map
\[\wedge\Phi_{2}^{\ast}\otimes H^{\ast}(\wedge\Phi_{1}^{\ast}):E_{1}^{\ast,\ast}(\g_{2})= \bigwedge\frak a_{2}^{\ast}\otimes H^{\ast} (\n_{2})\to\bigwedge\frak a_{1}^{\ast}\otimes  H^{\ast}( \n_{1})=E_{1}^{\ast,\ast}(\g_{1})\]
as in  Lemma \ref{lni}.
\end{lemma}

\section{Lefschetz coincidence numbers of Mostow solvamanifolds}

\begin{theorem}
Let $G_{1}$ and $G_{2}$ be simply connected solvable Lie groups of the same dimension  with lattices $\Gamma_{1}$ and $\Gamma_{2}$.
We assume they satisfy the Mostow condition.
Let $f,g: G_{1}/\Gamma_{1}\to G_{2}/\Gamma_{2}$ be  continuous maps.
Take linearizations $\Phi_{1}$,   $\Phi_{2}$  of $f$ and $\Psi_{1}$,  $\Psi_{2}$  of $g$ as Section \ref{linnn}.
Take representation matrices $A_{1}$, $A_{2}$, $B_{1}$ and $B_{2}$ of  $\Phi_{1\ast} $, $\Phi_{2\ast} $, $\Psi_{1\ast} $ and $\Psi_{2\ast} $ associated with basis of Lie algebras.
Let $A=A_{1}\oplus A_{2}$ and $B=B_{1}\oplus B_{2}$.
Then we have
\[L(f,g)={\rm det}(A-B).
\]

\end{theorem}

\begin{proof}
By Lemma \ref{spel}, we have
\[L(f,g)=L({\rm Tot}^{\ast}E^{\ast,\ast}_{2}(f),{\rm Tot}^{\ast}E^{\ast,\ast}_{2}(g)).
\]
By Lemma \ref{sem} and the Hopf lemma, we have
\[L({\rm Tot}^{\ast}E^{\ast,\ast}_{2}(f),{\rm Tot}^{\ast}E^{\ast,\ast}_{2}(g))=L(\wedge\Phi_{2}^{\ast}\otimes\wedge\Phi_{1}^{\ast}, \wedge\Psi_{2}^{\ast}\otimes\wedge\Psi_{1}^{\ast}).
\]
Take  bases $\{X^{1}_{1},\dots, X^{1}_{n}\}$, $\{Y_{1}^{1},\dots,Y_{m}^{1}\}$, $\{ X_{1}^{2},\dots, X_{n^{\prime}}^{2}\}$ and $\{Y_{1}^{2},\dots, Y_{m^{\prime}}^{2}\}$ of $\n_{1}$, $\frak a_{1}$, $\n_{2}$ and $\frak a_{2}$ which give representation matrices $A_{1}$, $A_{2}$, $B_{1}$ and $B_{2}$ of  $\Phi_{1\ast} $, $\Phi_{2\ast} $, $\Psi_{1\ast} $ and $\Psi_{2\ast} $ respectively.
Consider the dual bases $\{x^{1}_{1},\dots, x^{1}_{n}\}$, $\{y_{1}^{1},\dots,y_{m}^{1}\}$, $\{x_{1}^{2},\dots, x_{n^{\prime}}^{2}\}$ and $\{ y_{1}^{2},\dots, y_{m^{\prime}}^{2}\}$ of these bases respectively
Then we have
\[\bigwedge\frak a_{1}^{\ast}\otimes \bigwedge \n_{1}^{\ast}=\bigwedge \langle x^{1}_{1},\dots, x^{1}_{n}, y_{1}^{1},\dots,y_{m}^{1} \rangle,
\]
\[\bigwedge\frak a_{2}^{\ast}\otimes \bigwedge \n_{2}^{\ast}=\bigwedge \langle x^{2}_{1},\dots, x^{2}_{n}, y_{1}^{2},\dots,y_{m}^{2} \rangle
\]
and the maps  $\wedge\Phi_{2}^{\ast}\otimes\wedge\Phi_{1}^{\ast}$ and $\wedge\Psi_{2}^{\ast}\otimes\wedge\Psi_{1}^{\ast}$ are represented by $\wedge A^{\ast}$ and $\wedge B^{\ast}$ respectively.
 Hence we have 
\[L(f,g)=L(\wedge\Phi_{2}^{\ast}\otimes\wedge\Phi_{1}^{\ast}, \wedge\Psi_{2}^{\ast}\otimes\wedge\Psi_{1}^{\ast})=L(\wedge A^{\ast}, \wedge B^{\ast}).
\]
By Theorem \ref{HL}, we have
\[L(\wedge A^{\ast}, \wedge B^{\ast})={\rm det}(A^{\ast}-B^{\ast})={\rm det}(A-B).
\]
Hence the theorem follows.
\end{proof}

\end{document}